\documentclass[a4paper]{amsart}
\usepackage{amsthm}
\usepackage{amssymb}
\usepackage{enumerate}
\usepackage{verbatim}
\usepackage{tikz-cd} 
\usepackage{hyperref}

\numberwithin{equation}{section}

\theoremstyle{plain}

\newtheorem{lemma}{Lemma}[section]
\newtheorem{theorem}[lemma]{Theorem}

\newtheorem{corollary}[lemma]{Corollary}

\theoremstyle{definition}

\theoremstyle{remark} 
\newtheorem{remark}[lemma]{Remark}

\newcommand{\Adloc}{\operatorname{\mathsf{Adloc}}}

\renewcommand{\dim}{\operatorname{dim}}
\newcommand{\End}{\operatorname{End}}
\newcommand{\Ext}{\operatorname{Ext}}

\newcommand{\Hom}{\operatorname{Hom}}

\renewcommand{\mod}{\operatorname{\mathsf{mod}}}

\newcommand{\Proj}{\operatorname{Proj}}

\newcommand{\res}{\operatorname{res}}

\newcommand{\supp}{\operatorname{supp}}

\newcommand{\tExt}{\operatorname{\rlap{$\smash{\widehat{\mathrm{Ext}}}$}\phantom{\mathrm{Ext}}}}
\newcommand{\Thick}{\operatorname{\mathsf{Thick}}}

\def\mcU{\mathcal{U}}

\def\sfC{\mathsf C}

\def\bbN{\mathbb N}

\def\bbZ{\mathbb Z}

\newcommand{\fp}{\mathfrak{p}}
\newcommand{\fq}{\mathfrak{q}}

\title{The variety of subadditive functions for finite group schemes}

\author[Benson and Krause]{Dave Benson and Henning Krause}

\address{Dave Benson \\ 
Institute of Mathematics\\ 
University of Aberdeen\\ 
King's College\\ 
Aberdeen AB24 3UE\\ 
Scotland U.K.}

\address{Henning Krause\\ 
Fakult\"at f\"ur Mathematik\\ 
Universit\"at Bielefeld\\ 
33501 Bielefeld\\ 
Germany.}

\begin{document}

\begin{abstract} 
  For a finite group scheme, the subadditive functions on finite
  dimensional representations are studied. It is shown that the
  projective variety of the cohomology ring can be recovered from the
  equivalence classes of subadditive functions. Using Crawley-Boevey's
  correspondence between subadditive functions and endofinite modules,
  we obtain an equivalence relation on the set of point modules
  introduced in our joint work with Iyengar and Pevtsova. This
  corresponds to the equivalence relation on $\pi$-points introduced
  by Friedlander and Pevtsova.
\end{abstract}


\date{April 5, 2016}

\maketitle

\section{Introduction}

A theorem of Crawley-Boevey \cite{CB1992} gives a correspondence
between endofinite modules for a ring $A$ and subadditive functions on
the finitely presented $A$-modules.

We examine this correspondence in the context of finite group schemes,
for certain endofinite `point modules' for a finite group scheme $G$
introduced in our joint work with Iyengar and Pevtsova
\cite{Benson/Iyengar/Krause/Pevtsova:2015c}. These point modules come
from the $\pi$-points introduced by Friedlander and Pevtsova
\cite{Friedlander/Pevtsova:2007a}.  There is a natural equivalence
relation on $\pi$-points, and it is proved in
\cite{Friedlander/Pevtsova:2007a} that the equivalence classes of
$\pi$-points can be used to reconstruct the variety $\Proj H^*(G,k)$.

We translate the equivalence relation of Friedlander and Pevtsova into
a corresponding equivalence relation on subadditive functions.  This
enables us to prove that one can recover $\Proj H^*(G,k)$ from the
subadditive functions on finite dimensional $G$-modules in a natural
way (Theorem~\ref{th:main}).

\section{Subadditive functions and endofinite modules}

We briefly review Crawley-Boevey's correspondence between subadditive 
functions and endofinite modules. 

We fix a ring $A$ and consider the category of (right) $A$-modules.
Let $\mod A$ denote the full subcategory of finitely presented
$A$-modules. For an $A$-module $M$ let $\ell_A(M)$ denote its
composition length.

A \emph{subadditive function} $\chi\colon\mod A\to\bbN$ assigns to
each finitely presented $A$-module a non-negative integer such that
\begin{enumerate}[\quad\rm(1)]
\item $\chi(X\oplus Y)=\chi(X)+\chi(Y)$ for all $X,Y\in\mod A$, and 
\item $\chi(X)+\chi(Z)\ge\chi(Y)$ for each exact sequence $X\to Y\to
  Z\to 0$ in $\mod A$.
\end{enumerate}
A subadditive function $\chi\neq 0$ is \emph{irreducible} if $\chi$
cannot be written as a sum of two non-zero subadditive functions.

An $A$-module $M$ is called \emph{endofinite} if it has finite
composition length when viewed as a left module over its endomorphism
ring $\End_A(M)$.  An endofinite $A$-module $M$ gives rise to a
subadditive function $\chi_M$ by setting
\[\chi_M(X):=\ell_{\End_A(M)}(\Hom_A(X,M))\quad\text{for} \quad X\in\mod A.\]

The following theorem of Crawley-Boevey \cite[\S 5]{CB1992} provides the
context for our study of subadditive functions.

\begin{theorem}\label{th:CB}
  Any subadditive function $\mod A\to \bbN$ can be written uniquely as
  a finite sum of irreducible subadditive functions. Sending an
  endofinite $A$-module $M$ to $\chi_M$ induces a bijection between
  the isomorphism classes of indecomposable endofinite $A$-modules and
  the irreducible subadditive functions $\mod A\to \bbN$.\qed
\end{theorem}

Note that every endofinite module decomposes uniquely into
indecomposable endofinite modules. We have
$\chi_{M\oplus M'}=\chi_M+\chi_{M'}$ when $M$ and $M'$ have no common
indecomposable summand, while $\chi_{M\oplus M}=\chi_M$ \cite[\S
5]{CB1992}.

\section{The additive locus of a subadditive function}

We fix a ring $A$ such that $\mod A$ is an abelian category. For a
subadditive function $\chi\colon\mod A\to\bbN$ we define
the \emph{additive locus} $\Adloc(\chi)$ as the full subcategory of
objects $Z\in\mod A$ such that for every exact sequence
$0\to X\to Y\to Z\to 0$ in $\mod A$
\[\chi(X) -\chi(Y) + \chi(Z)=0.\]

We collect some basic properties of the additive locus.

\begin{lemma}\label{le:adloc}
  Let $M$ be an endofinite $A$-module and $Z\in\mod A$. Then $Z$ belongs to
  $\Adloc(\chi_M)$ if and only if $\Ext_A^1(Z,M)=0$.
\end{lemma}
\begin{proof}
  Applying
  $\Hom_A(-,M)$ to an exact sequence $0\to X\to Y\to Z\to 0$ in
  $\mod A$ induces a long exact sequence
  \[0\to\Hom_A(Z,M)\to \Hom_A(Y,M)\to \Hom_A(X,M)\to \Ext^1_A(Z,M)\to
  \cdots\]
  of $\End_A(M)$-modules. Clearly, $\Ext_A^1(Z,M)=0$
  implies \[\chi_M(X) -\chi_M(Y) + \chi_M(Z)=0.\]
  The converse follows by choosing for $Y$ a projective $A$-module.
\end{proof}

\begin{lemma}\label{le:adloc-sum}
Let $\chi,\chi'\colon\mod A\to\bbN$ be subadditive functions. Then
\[\Adloc(\chi+\chi')=\Adloc(\chi)\cap\Adloc(\chi').\]
\end{lemma}
\begin{proof}
  Let $\eta\colon 0\to X\to Y\to Z\to 0$ be an exact sequence in
  $\mod A$. Then $\chi(\eta):=\chi(X) -\chi(Y) + \chi(Z)$ is a
  non-negative integer. Thus $(\chi+\chi')(\eta)=0$ if and only if
  $\chi(\eta)=0=\chi'(\eta)$. From this observation the assertion of the
  lemma follows.
\end{proof}

\begin{lemma}\label{le:adloc-mod}
Let $\chi\colon\mod A\to\bbN$ be a subadditive function. Then there
exists an endofinite $A$-module $M$ such that $\Adloc(\chi)
=\Adloc(\chi_M)$.
\end{lemma}
\begin{proof}
  We apply Theorem~\ref{th:CB}. Write $\chi=\sum_in_i\chi_i$ with
  $\chi_i$ irreducible and $\chi_i\neq\chi_j$ for all $i\neq j$. There
  are indecomposable endofinite $A$-modules $M_i$ such that
  $\chi_i=\chi_{M_i}$ for all $i$, and $\sum_i\chi_i=\chi_M$ for
  $M=\bigoplus_iM_i$. Then $\Adloc(\chi)=\Adloc(\chi_M)$ by
  Lemma~\ref{le:adloc-sum}.
\end{proof}

Let $\chi,\chi'\colon\mod A\to\bbN$ be subadditive functions. We set
\[\chi\ge\chi'\quad :\,\Longleftrightarrow\quad \Adloc(\chi)\subseteq
\Adloc(\chi')\] and call $\chi$ and $\chi'$ \emph{equivalent} if
$\Adloc(\chi)=\Adloc(\chi')$. Thus the equivalence classes of
subadditive functions form a partially ordered set.

\section{Subadditive functions for finite group schemes}

Let $G$ be a finite group scheme over a field $k$. Thus $G$ is an
affine group scheme such that its coordinate algebra $k[G]$ is finite
dimensional as a $k$-vector space. The $k$-linear dual of $k[G]$ is a
cocommutative Hopf algebra, called the group algebra of $G$, and
denoted $kG$. We identify $G$-modules with modules over the group
algebra $kG$. The category of finite dimensional $G$-modules is
denoted by $\mod G$.  

The tensor product over $k$ induces one for $\mod G$ via the diagonal
action of $G$. A subcategory $\sfC$ of $\mod G$ is called \emph{tensor
  closed}, if $M\in\sfC$ and $N\in\mod G$ imply that $M\otimes_k N$ is
in $\sfC$.

We write $H^*(G,k)$ for the cohomology algebra of $G$ and
$\Proj H^*(G,k)$ for the set of its homogeneous prime ideals not
containing the unique maximal ideal of positive degree elements. Note
that $H^*(G,k)$ acts on $\Ext^*_G(M,N)$ for all $G$-modules $M,N$.
The \emph{support} of a finite dimensional $G$-module $M$ is
\[\supp_G(M):=\{\fp\in\Proj H^*(G,k)\mid \Ext_G^*(M,M)_\fp\neq 0\}.\]

The following theorem says that $\Proj H^*(G,k)$ can be recovered from
the equivalence classes of subadditive functions on $\mod G$.

We call a subadditive function on $\mod G$ \emph{tensor
  closed} if its additive locus is a tensor closed subcategory of $\mod G$.

An element $x$ of a poset is \emph{join irreducible} if it is not the
supremum of elements that are strictly smaller than $x$.

\begin{theorem}\label{th:main}
  Let $G$ be a finite group scheme over a field $k$ and let $P(G,k)$
  denote the partially ordered set of equivalence classes of tensor
  closed subadditive functions $\chi\colon\mod G\to\bbN$. If the class
  of $\chi$ is join irreducible, then there exists a unique
  $\fp\in\Proj H^*(G,k)$ such that
\[\Adloc(\chi)=\{M\in\mod G\mid \fp\not\in\supp_G(M)\}.\]
Sending $\chi$ to $\fp$ induces an order isomorphism between the set
of join irreducible elements of $P(G,k)$ and $\Proj H^*(G,k)$.
\end{theorem}

The proof will be given in \S\ref{se:proof}.

\section{The additive locus of a tensor closed function}

In this section we study the additive locus of a subadditive function
that is tensor closed.  We need the following result, which
essentially comes from \cite{BCR1990}.

For $G$-modules $M,N$ and $n\in\bbZ$, we write $\tExt^n_G(M,N)$ for
the $n$th Tate extension group; it equals $H^n(\Hom_G(\mathbf t M,N))$
where $\mathbf t M$ denotes a Tate resolution of $M$. Note that
\[\tExt^n_G(M,N)\cong\Ext^n_G(M,N) \quad\text{for}\quad n>0.\]

\begin{theorem}\label{th:BCR}
Given a finite group scheme $G$ over a field $k$, there exists a
positive integer $r$ such that for any $G$-modules $M$ and $N$,
if $\tExt^n_G(M,N)=0$ for $r$ consecutive values of $n$ then
$\tExt^n_G(M,N)=0$ for all $n$ positive and negative.
\end{theorem}
\begin{proof}
The proof for finite groups given in \cite{BCR1990}
works just as well for finite group schemes. The input is the finite
generation of cohomology, which for finite group schemes was
proved by Friedlander and Suslin \cite{Friedlander/Suslin:1997a}.
\end{proof}

A full subcategory $\sfC$ of $\mod G$ is said to be \emph{thick}, if
any direct summand of a module in $\sfC$ is also in $\sfC$ and for
every exact sequence $0\to M'\to M\to M''\to 0$ in $\mod G$ with
two of $M, M', M''$ in $\sfC$ also the third is in $\sfC$.

\begin{corollary}\label{co:adloc}
  Let $\chi\colon\mod G\to\bbN$ be a tensor closed subadditive
  function. Then $\Adloc(\chi)$ is a thick subcategory of $\mod G$. If $\chi=\chi_M$
for some  endofinite $G$-module $M$, then we have
\[\Adloc(\chi_M)=\{X\in\mod G\mid \tExt^*_G(X,M)=0\}.\]
\end{corollary}
\begin{proof}
  We may assume that $\chi=\chi_M$ for some endofinite $G$-module $M$
  by Lemma~\ref{le:adloc-mod}. Fix $X\in\mod G$. From
  Lemma~\ref{le:adloc} it follows that $X$ is in $\Adloc(\chi)$ if and
  only if $\Ext_G^1(X,M)=0$. Let $\Omega(X)$ denote the kernel of a
  projective cover of $X$, and observe that
  $\Omega(X)\cong X\otimes_k\Omega(k)$ up to projective direct
  summands.  Using dimension shift and the fact that $\chi$ is tensor
  closed, it follows that $X$ is in $\Adloc(\chi)$ if and only if
  $\Ext_G^n(X,M)=0$ for every $n>0$.  Now Theorem~\ref{th:BCR} implies
  that $X$ is in $\Adloc(\chi)$ if and only if
  $\tExt_G^*(X,M)=0$. From this description of $\Adloc(\chi)$ and the
  long exact sequence for $\tExt_G^*(-,M)$ it follows that
  $\Adloc(\chi)$ is a thick subcategory.
\end{proof}

\begin{remark}\label{re:tensor}
Given any subadditive function $\chi\colon\mod G\to\bbN$, then
\[\chi':=\sum_{S\text{ simple}}\chi(-\otimes_k S)\] is up
to equivalence the unique minimal tensor closed subadditive function 
such that $\chi'\geq\chi$. In fact, Lemma~\ref{le:adloc-sum} shows
that \[\Adloc(\chi')=\{X\in\mod G\mid X\otimes_k Y\in\Adloc(\chi)\text{
  for all }Y\in\mod G\}.\]
\end{remark}

\section{Tensor closed thick subcategories  and $\pi$-points}

In this section we recall some of the results from our joint work with
Iyengar and Pevtsova
\cite{Benson/Iyengar/Krause/Pevtsova:2015b,Benson/Iyengar/Krause/Pevtsova:2015c}.
The first result is a classification of tensor closed thick
subcategories of $\mod G$ that has been anticipated in
\cite{Friedlander/Pevtsova:2007a}.

For a subcategory $\sfC$ of $\mod G$ we set 
\[\supp_G(\sfC):=\bigcup_{M\in\sfC}\supp_G(M).\]

A subset $\mcU$ of $\Proj H^*(G,k)$ is called \emph{specialisation
  closed} if whenever $\fp$ is in $\mcU$ so is any prime $\fq$
containing $\fp$. 
\begin{theorem}\label{th:bikp}
Let $G$ be a finite group scheme over a field $k$.
Then the assignments
\[\sfC\mapsto \supp_G(\sfC)\qquad\text{and}\qquad
\mcU\mapsto\{M\in\mod G\mid \supp_G(M)\subseteq\mcU\}\]
give mutually inverse and inclusion preserving bijections between the
non-zero tensor closed thick subcategories of $\mod G$ and the
specialisation closed subsets of $\Proj H^*(G,k)$.
\end{theorem}
\begin{proof}
 See Theorem~10.3 of \cite{Benson/Iyengar/Krause/Pevtsova:2015b}.
\end{proof}

In \cite{Friedlander/Pevtsova:2007a}, Friedlander and Pevtsova
introduced for a finite group scheme $G$ over a field $k$ of
characteristic $p>0$ the notion of a \emph{$\pi$-point}. This is by
definition a flat algebra homomorphism $\alpha\colon K[t]/(t^p)\to KG$
for some field extension $K/k$ such that $\alpha$ factors through the
group algebra of a unipotent abelian subgroup scheme of $G_K$, where
$G_K$ denotes the group scheme over $K$ with group algebra
$KG:=kG\otimes_kK$. Let $\alpha^*\colon\mod G_K\to \mod K[t]/(t^p)$
denote restriction along $\alpha$.  We set
\[\Thick(\alpha):=\{M\in\mod G\mid \alpha^*(M\otimes_k K)\text{ is projective}\}\]
and observe that $\Thick(\alpha)$ is a tensor closed thick subcategory
of $\mod G$. 

Two $\pi$-points $\alpha$ and $\beta$ are
\emph{equivalent} if $\Thick(\alpha)=\Thick(\beta)$, and the
equivalence classes are in natural bijection with the points of
$\Proj H^*(G,k)$; see
\cite[Theorem~3.6]{Friedlander/Pevtsova:2007a}. The following theorem
makes this correspondence explicit.

\begin{theorem}\label{th:bikp-pipoint}
Let  $\fp\in\Proj H^*(G,k)$. Then
there exists a  $\pi$-point $\alpha$ such that
\[\Thick(\alpha)=\{M\in\mod G\mid \fp\not\in\supp_G(M)\}.\]
\end{theorem}
\begin{proof}
 See Theorem~6.1 of \cite{Benson/Iyengar/Krause/Pevtsova:2015b}.
\end{proof}

\section{Proof of the main theorem}\label{se:proof}

This section provides the proof of the main theorem, and we begin with
some preparation.

Let $\alpha\colon K[t]/(t^p)\to KG$ be a $\pi$-point for $G$. This gives rise
to a subadditive function $\chi_\alpha$ on $\mod G$ by setting
\[\chi_\alpha:=\dim_K(\Hom_{K[t]/(t^p)}(\alpha^*(-\otimes_k K),K)).\]
We may think of $\chi_\alpha$ as composite
\[
\begin{tikzcd}
\mod G\arrow{r}{-\otimes_k K}&\mod G_K\arrow{r}{\alpha^*}&\mod
K[t]/(t^p)\arrow{r}{\chi_K}&\bbN.
\end{tikzcd}
\]

\begin{lemma}\label{le:pipoint}
We have $\Adloc(\chi_\alpha)=\Thick(\alpha)$.
\end{lemma}
\begin{proof}
  Let $Z\in\mod G$. If $\alpha^*(Z\otimes_k K)$ is a projective
  $K[t]/(t^p)$-module, then for any exact sequence
  $0\to X\to Y\to Z\to 0$ in $\mod G$ we have
 \[\chi_\alpha(X) -\chi_\alpha(Y) + \chi_\alpha(Z)=0.\]
 For the converse, choose an exact sequence $0\to X\to Y\to Z\to 0$ in
 $\mod G$ with $Y$ projective. Thus $\alpha^*(Y\otimes_k K)$ is
 projective. If
\[\chi_\alpha(X) -\chi_\alpha(Y) + \chi_\alpha(Z)=0,\]
then the sequence 
\[0\to \alpha^*(X\otimes_k K)\to \alpha^*(Y\otimes_k K)\to
\alpha^*(Z\otimes_k K)\to 0\]
splits, and therefore $\alpha^*(Z\otimes_k K)$ is projective.
\end{proof}

\begin{lemma}\label{le:top}
Let $X$ be a topological space which is a $T_0$ space. Fix a set
$P$ of closed subsets of $X$ that contains the closure
$\overline{\{x\}}$ for each $x\in X$, and view $P$ as a poset via the
inclusion order. Then the assignment $x\mapsto \overline{\{x\}}$
identifies $X$ with the set of join irreducible elements of $P$.
\end{lemma}
\begin{proof}
Straightforward.
\end{proof}

\begin{proof}[Proof of Theorem~\ref{th:main}]
  We consider $\Proj H^*(G,k)$ with the Hochster dual of the Zariski
  topology. Thus the open subsets are precisely the specialisation
  closed subsets \cite{Hochster:1969a}.  The assignment
\[\chi\longmapsto \Proj H^*(G,k)\setminus\supp_G(\Adloc(\chi))\]
identifies the equivalence classes of tensor closed subadditive
functions $\mod G\to\bbN$ with certain closed subsets of
$\Proj H^*(G,k)$. This follows from Corollary~\ref{co:adloc} and
Theorem~\ref{th:bikp}. On the other hand,
Theorem~\ref{th:bikp-pipoint} provides for 
each $\fp\in\Proj H^*(G,k)$ a $\pi$-point $\alpha$ such that
$\chi_\alpha$ is a tensor closed subadditive function
satisfying
\[\Adloc(\chi_\alpha)=\{M\in\mod G\mid
\fp\not\in\supp_G(M)\}\]
by Lemma~\ref{le:pipoint}.  Thus the function $\chi_\alpha$ is sent to
the closure $\overline{\{\fp\}}=\{\fq\mid\fq\subseteq\fp\}$. Moreover,
given a
$\pi$-point $\beta$ corresponding to $\fq\in\Proj H^*(G,k)$ we have
\[\chi_\alpha\ge
\chi_\beta\quad\Longleftrightarrow\quad \Adloc(\chi_\alpha)\subseteq
\Adloc(\chi_\beta)\quad \Longleftrightarrow\quad\fp\supseteq\fq.\]
Now the assertion of the theorem follows from Lemma~\ref{le:top}.
\end{proof}

\section{$\pi$-points and point modules}\label{se:point}

Let $G$ be a finite group scheme over a field $k$ of characteristic
$p>0$.  To each $\pi$-point $\alpha\colon K[t]/(t^p)\to KG$
corresponds a \emph{point
  module} \[\Delta_G(\alpha):=\res^K_k(\Hom_{K[t]/(t^p)}(KG,K)).\]
This is an endofinite $G$-module and plays a prominent role in recent
work with Iyengar and Pevtsova
\cite{Benson/Iyengar/Krause/Pevtsova:2015c}.

\begin{lemma}\label{le:chi-pipoint}
We have  $\chi_{\Delta_G(\alpha)}=\chi_\alpha$. 
\end{lemma}
\begin{proof}
Adjunction gives for each $M\in\mod G$ a natural isomorphism
\begin{align*}
\Hom_{G}(M,\Delta_G(\alpha))&\cong \Hom_{G_K}(M\otimes_k
                               K,\Hom_{K[t]/(t^p)}(KG,K))\\
&\cong\Hom_{K[t]/(t^p)}(\alpha^*(M\otimes_kK),K)
\end{align*}
which restricts to submodules over the endomorphisms rings
of $\Delta_G(\alpha)$ and $K$ respectively. 
\end{proof}

\begin{corollary}\label{co:main}
  Let $G$ be a finite group scheme over a field $k$. Given
  $\pi$-points $\alpha$ and $\beta$ of $G$, the following conditions
  are equivalent:
\begin{enumerate}[\quad\rm(1)]
\item The $\pi$-points $\alpha$ and $\beta$ are equivalent.
\item The subadditive functions $\chi_{\Delta_G(\alpha)}$ and
  $\chi_{\Delta_G(\beta)}$ are equivalent.
\end{enumerate}
\end{corollary}
\begin{proof}
Combine Lemmas~\ref{le:pipoint} and \ref{le:chi-pipoint}.
\end{proof}


\begin{thebibliography}{99}
%
\bibitem{BCR1990} D. J. Benson, J. F. Carlson\ and\ G. R. Robinson, \emph{On
  the vanishing of group cohomology}, J. Algebra {\bf 131} (1990),
  no.~1, 40--73
%
\bibitem{Benson/Iyengar/Krause/Pevtsova:2015b} D.~J. Benson,
  S.~B. Iyengar, H.~Krause, and J.~Pevtsova, \emph{Stratification for
    module categories of finite group schemes}, preprint 2015;
  \texttt{arXiv:1510.06773}.
%
\bibitem{Benson/Iyengar/Krause/Pevtsova:2015c} D.~J. Benson,
  S.~B. Iyengar, H.~Krause, and J.~Pevtsova, \emph{Colocalising
    subcategories of modules over finite group schemes}, preprint
  2016;  \texttt{arXiv:1604.00524}.
%
\bibitem{CB1991} W. Crawley-Boevey, \emph{Tame algebras and generic modules},
  Proc. London Math. Soc. (3) {\bf 63} (1991), no.~2, 241--265.
%
\bibitem{CB1992} W. Crawley-Boevey, \emph{Modules of finite length
    over their endomorphism rings}, Representations of algebras and
  related topics (Tsukuba, 1990), 127--184, London Math. Soc. Lecture
  Note Ser., 168 Cambridge Univ. Press, Cambridge, 1992.
%
\bibitem{Friedlander/Pevtsova:2007a} E.~M. Friedlander and
  J.~Pevtsova, \emph{$\Pi$-supports for modules for finite groups
      schemes}, Duke Math.\ J. \textbf{139} (2007), 317--368.
%
\bibitem{Friedlander/Suslin:1997a} E. M. Friedlander and A. Suslin,
  \emph{Cohomology of finite group schemes over a field}, Invent. Math. \textbf{127}
  (1997), 209–270.
%
\bibitem{Hochster:1969a} M. Hochster, \emph{Prime ideal structure in
    commutative rings}, Trans. Amer. Math. Soc. \textbf{142} (1969),
  43--60.
%
\end{thebibliography}
\end{document}